\documentclass{article}

\usepackage{amsmath}
\usepackage{amsfonts}
\usepackage{amssymb}
\usepackage{amsthm}
\usepackage{latexsym}
\usepackage[all]{xy}
\usepackage{eufrak}
\usepackage{indentfirst}
\usepackage{color}

\usepackage[pagewise]{lineno}

\theoremstyle{plain}
\newtheorem{theorem}{Theorem}[section]
\newtheorem{lemma}[theorem]{Lemma}
\newtheorem{proposition}[theorem]{Proposition}
\newtheorem{corollary}[theorem]{Corollary}

\theoremstyle{definition}
\newtheorem{definition}[theorem]{Definition}
\newtheorem{example}[theorem]{Example}
\newtheorem{remark}[theorem]{Remark}

\newcommand\bt{\begin{theorem}}
\newcommand\et{\end{theorem}}

\newcommand\bl{\begin{lemma}}
\newcommand\el{\end{lemma}}

\newcommand\bp{\begin{proposition}}
\newcommand\ep{\end{proposition}}

\newcommand\bc{\begin{corollary}}
\newcommand\ec{\end{corollary}}

\newcommand\bd{\begin{definition}}
\newcommand\ed{\end{definition}}

\newcommand\br{\begin{remark}}
\newcommand\er{\end{remark}}

\newcommand\bex{\begin{example}}
\newcommand\eex{\end{example}}

\newcommand\bess{\begin{eqnarray*}}
\newcommand\eess{\end{eqnarray*}}

\title{Coperfectly Hopfian Groups and Shape Fibrator's Properties\footnote{This research did not receive any specific grant from funding agencies in the public, commercial, or
not-for-profit sectors.}}

\author{Violeta Vasilevska\\
{\it Department of Mathematics}\\
{\it  Utah Valley University}\\
{ \it 800 West University Parkway}\\
{\it  Orem, UT 84058, USA}\\
{\it Phone}: (801) 863-8649; {\it Fax}: (801) 863-6254\\
 {\it  e-mail: Violeta.Vasilevska@uvu.edu}}

\begin{document}
\maketitle
\begin{abstract}

This paper provides further investigation of the concept of shape m$_{\rm simpl}$-fibrators (previously introduced by the author). The main results identify shape m$_{\rm simpl}$-fibrators among direct products of Hopfian manifolds. First it is established that every closed orientable manifold homotopically determined by $\pi_1$ with coperfectly Hopfian group (a new class of Hopfian groups that are introduced here) is a shape m$_{\rm simpl}$o-fibrator if it is a codimension-2 fibrator (Theorem \ref{Fib}). The main result (Theorem \ref{Dir}) states that the direct product of two closed orientable manifolds (of different dimension) homotopically determined by $\pi_1$ and  with coperfectly Hopfian fundamental groups (one normally incommensurable with the other one) is a shape m$_{\rm simpl}$o-fibrator, if it is a Hopfian manifold and a codimension-2 fibrator.

\end{abstract}
\vspace{0.5cm}


\vspace{0.5cm}

 {\it AMS (MOS) Subj. Class.}: Primary 57N15; 57M07,
Secondary 57N25; 54B15.
\newpage

\section{\bf Introduction}

This paper continues an investigation of the proper mappings from $(n+k)$-manifolds onto triangulated manifolds that have closed manifolds as point pre-images in the PL setting introduced in \cite{VV1}.

The approximate fibrations, introduced by Coram and Duvall \cite{Coram 1, Coram 2}, are proper mappings that satisfy an approximate version
of the homotopy lifting property - the defining property of the more familiar class of fibrations. They form an important class of mappings mostly because of their nice properties. Among them the most useful property is the existence of an exact sequence involving  the homotopy groups of domain, target, and shape-theoretical homotopy groups of any point inverse of $p$. Note that these properties of an approximate fibration reduce to the usual properties of Hurewicz fibration when working with a PL approximate fibration, because the fibers are ANRs, so the $i^{\rm th}$ shape homotopy groups are isomorphic to $i^{\rm th}$ homotopy groups.

Sometimes a proper map defined on an arbitrary manifold of a specific dimension can be recognized as an approximate fibration due to having point inverses all of a certain homotopy type (or shape). Hence, in order to recognize manifolds that can force a proper map to be an approximate fibrations (when they appear as point pre-images of the map), Daverman introduced the concept of {\it codimension-$k$ (orientable) fibrator} \cite{D.Dec} and later the concept of {\it PL (orientable) fibrator} \cite{D.PL}. In \cite{VV1} the author introduced the concept of {\it codimension-$k$ shape $m_{simpl}(o)$-fibrator} (and more generally the concept of {\it shape $m_{simpl}(o)$-fibrator}) as PL fibrators in a slightly different PL setting than the one used by Daverman in \cite{D.PL}, and provided examples of manifolds that are shape m$_{\rm simpl}$o-fibrators. In addition, in \cite{VV2} the author provided examples of manifolds that are codimension-($k+1$) shape m$_{\rm simpl}$o-fibrators ($k \ge 2$).

The following is the main question that we address in this paper: {\it Which direct products of Hopfian manifolds are shape m$_{simpl}$o-fibrators?}
The question of whether the collection of codimension-$k$ PL (or shape m$_{\rm simpl}$) fibrators is closed under taking Cartesian product remains unsolved, but seems not likely (because of the examples presented in \cite{D.cod2}). Some partial answers to this question for codimension-$k$ PL fibrators (as well as PL fibrators) have been given in \cite{Daverman-Im-Kim, Yo1, IKW, IKW1, Im-Kim3, JK}.

In this paper, we provide examples of shape m$_{\rm simpl}$o-fibrators among direct products of Hopfian manifolds. Note that analysis of fibrator properties applies mostly to Hopfian manifolds with Hopfian fundamental groups,
hence in search for shape m$_{\rm simpl}$o-fibrators among products of Hopfian manifolds, first we need to look for a particular type of Hopfian groups (the ones that are closed under taking Cartesian products).

Therefore, this paper has two parts. The first part, Sections \ref{PHG} and \ref{PPHG}, introduce and discuss two group properties (coperfectly Hopfian group and normal incommensurability of groups) that are needed to provide closure under taking direct product of Hopfian groups (see Theorem \ref{PH} below). Section \ref{PHG} provides examples of coperfectly Hopfian groups among finite and infinite groups, including the fundamental groups of closed orientable surfaces with genus $g >1$ (see Theorem \ref{PH-S} below). In addition, this section lists information about normal incommensurability of groups (e.g., finitely generated groups with $h$ generators are normally incommensurable with the fundamental group of closed orientable surfaces with genus $g >1$, if $h < 2g$ [see Theorem \ref{fg} below]). Section \ref{PPHG} discusses conditions under which free products (see Corollary \ref{free product 3} below) and direct products (see Theorem \ref{PH} below) of coperfectly Hopfian groups are coperfectly Hopfian.

 The second part of the paper, Sections \ref{SF} and \ref{PSF}, provide applications to shape m$_{\rm simpl}$o-fibrators. Namely, Section \ref{SF} delivers examples of shape m$_{\rm simpl}$o-fibrators among codimenion-2 fibrators who are closed orientable manifolds homotopically determined by $\pi_1$ with coperfectly Hopfian fundamental groups (see Theorem \ref{Fib} below). Section \ref{PSF} contains the main results that provide detection of shape m$_{\rm simpl}$o-fibrators among direct products of Hopfian manifolds (see Theorems \ref{Dir} and \ref{C-Dir} below).

\section{\bf Definitions and notations}

Throughout the paper, symbols $\cong$ and $\chi$ will denote isomorphism and Euler characteristic respectively, and homology and
cohomology groups will be computed with integer coefficients. \cite{R-S} contains the terminology and definitions used for the material on piecewise-linear topology. Space means topological space and
maps are continuous functions. We assume that all spaces are locally
compact ANR. A {\it manifold} is assumed to be connected, metric,
and boundaryless. A manifold $M$ is {\it aspherical} if $\pi_i(M)=0$ for all $i>1$. If
$M$ is a manifold then $M^n$ will denote a manifold of dimension
represented by the superscript.

A {\it generalized $k$-manifold} is a finite dimensional, locally contractible metric space $X$, such that $H_*(X,X \backslash \{x\}) \cong H_*(\mathbb{R}^k, \mathbb{R}^k \backslash \{0\})$ for all $x \in X$. A {\it simplicial homotopy $k$-manifold} is a triangulated polyhedron $K$ in which the link of each $i$-simplex has the homotopy type of the $(k-i-1)$-sphere. Note that simplicial homotopy manifolds are genuine topological manifolds, unlike the polyhedral generalized manifolds, in which vertices possibly fail to have a Euclidean neighborhood. If $B$ is a simplicial complex, then $B^{(j)}$ denotes the
$j$-skeleton of $B$ and $B^j$ denotes the $j$-th derived subdivision
of $B$.

 A map $f:N \to
N'$ between closed orientable $n$-manifolds is said to have
(absolute) {\it degree} $d$ if there are choices of generators
$\gamma \in H_n(N) \cong \mathbb{Z}$, $\gamma' \in H_n(N') \cong
\mathbb{Z}$, such that $f_*(\gamma)=d\gamma'$, where $d \ge 0$ is an
integer. The {\it Hopfian  manifold}
\cite{D.MPL} is a closed orientable manifold such that every
degree one self-map which induces a $\pi_1$-isomorphism is a
homotopy equivalence. Examples of Hopfian manifolds include: every
closed orientable $n$-manifold that (1) is simply connected; or (2)
has a finite fundamental group; or (3) has a Hopfian fundamental group
and $n \le 4$ \cite{Ha}.
A manifold $N$ is {\it homotopically determined by $\pi_1$}
\cite{Daverman-Kim} if every self map $f:N\to N$ that induces a
$\pi_1$-isomorphism is a homotopy equivalence. Aspherical manifolds
are common examples of manifolds determined by $\pi_1$. No closed
$n$-manifold, $n>1$, with free fundamental group is
homotopically determined by $\pi_1$. Additional examples are
presented in \cite{Daverman-Kim}.

A proper surjective map $p:E \to B$ between locally compact ANR's
is an {\it approximate fibration} if $p$ satisfies the following
approximate homotopy lifting property: for an arbitrary space $X$,
and given a cover $\mathfrak{U}$ of $B$ and maps $g:X \to E$ and $H:
X \times [0,1] \to B$ such that $pg=H_0$, there exists a map
$\widetilde{H}: X \times [0,1] \to E$ such that $\widetilde{H}_0=g$
and $p\widetilde{H}$ and $H$ are $\mathfrak{U}$-close (i.e., for each
$z \in X \times [0,1]$, there exists $U_z \in \mathfrak{U}$ such
that $\{H(z), \, p\widetilde{H}(z)\} \subset U_z$).

Following P. Hall, we call a group $G$ {\it residually
finite} if to each non-identity element $g$ in $G$, there corresponds a
homomorphism taking $G$ onto a finite group and $g$ onto a non-identity
element of this image group. In other words, $G$ is a residually
finite group if every non-identity element of $G$ is mapped
nontrivially in some finite quotient group of $G$.

Recall that a group $G$ is {\it Hopfian} (after Heinz Hopf,
1894-1971) if every  epimorphism $\varphi:G \to G$ is an
automorphism. In other words, $G$ is Hopfian if it is not isomorphic
to a proper factor of itself. A group $G$ is {\it
hyper-Hopfian} \cite {D.HG} if every homomorphism $\varphi:G \to G$ with $\varphi
(G) \vartriangleleft G$ and $G/\varphi (G)$ cyclic is necessarily an
automorphism. A group $G$ is  {\it ultra-Hopfian} \cite{VV1} if  every nontrivial homomorphism $\varphi:G \to G$ with $\varphi(G) \unlhd G$ is an automorphism.

\section{\bf Coperfectly Hopfian Groups and Group Incommensurability}\label{PHG}

In this section we introduce and discuss two new group theoretical properties.

A group $G$ is called {\it coperfectly Hopfian} if  every homomorphism $\varphi:G \to G$ with $\varphi(G) \unlhd G$ and $G/{\varphi(G)}$ perfect, is an automorphism.
First note that coperfectly Hopfian groups are Hopfian groups by
definition and that no perfect group can be coperfectly Hopfian. Also, all ultra-Hopfian groups that are not perfect are coperfectly Hopfian.
Furthermore, all non-perfect simple groups are coperfectly Hopfian (since they do not have a proper normal subgroup isomorphic to a factor group of itself). Note that the simple groups $\mathbb{Z}_p$, $p$-prime, are examples of coperfectly Hopfian and
ultra-Hopfian groups that are not hyper-Hopfian groups.

\begin{theorem} \label{prop1} All Hopfian solvable groups are coperfectly Hopfian.
\end{theorem}

    \begin{proof} Let $G$ be a Hopfian solvable group and let $\varphi:G \to G$ be such that $\varphi(G)$ is a normal subgroup of $G$ and $G/{\varphi(G)}$ is perfect. Since $\varphi(G)$  is solvable (as a homomorphic image of the solvable group $G$), it follows that $G/{\varphi(G)}$ is a solvable group too. No nontrivial solvable group is perfect, hence $G/{\varphi(G)}$ must be trivial. Therefore $\varphi$ is surjective and the Hopfian property of $G$ implies that $\varphi$ is an isomorphism.
    \end{proof}

\begin{corollary} \label{Ab} All finitely generated Abelian groups are coperfectly Hopfian groups.
\end{corollary}

Theorem \ref{prop1} also implies that all groups of order less than 60 and finite groups of odd order are coperfectly Hopfian, since they are Hopfian solvable groups.

Recall that a {\it polycyclic group} is both a solvable group and a Noetherian group.

\begin{corollary} Every polycyclic group is coperfectly Hopfian.
\end{corollary}

    \begin{proof} This follows from Theorem \ref{prop1}, since all polycyclic groups are finitely generated residually finite groups by \cite[Theorem~ 3]{HIV}, hence Hopfian \cite{Mal}.
    \end{proof}

\begin{corollary}\label{Nil}  Every finitely generated nilpotent group is coperfectly Hopfian.
\end{corollary}

In particular, all finite $p$-groups are coperfectly Hopfian, since they are nilpotent.

By Theorem \ref{prop1} and Burnside's Theorem, the following result is easy seen.

\begin{corollary} \label{pq} Every group of order $p^nq^m$,
where $p, \, q$ are primes and $n, \, m$ are non-negative integers, is coperfectly Hopfian.
\end{corollary}

Dihedral groups $D_{2n+1}=\left\langle x,y \, \arrowvert \,
x^2=y^{2n+1}=1, \, x^{-1}yx=y^{-1}\right\rangle$ of order $2(2n+1)$,
where $2n+1$ is a prime,  are coperfectly Hopfian by Corollary
\ref{pq}. They are also hyper-Hopfian (see \cite[Section~ 4]{D.HG}) and ultra-Hopfian by \cite[Proposition~ 2.1]{VV1} groups as well. Furthermore, $D_{2^{n+1}}=\left\langle x,y \, \arrowvert \, x^2=y^{2^{n+1}}=1,
\, x^{-1}yx=y^{-1}\right\rangle$ are 2-groups, so coperfectly Hopfian by Corollary \ref{Nil}. Note that $D_{2^{n+1}}$ are not ultra-Hopfian (see \cite[Section~ 2]{VV1}). The quaternionic group $Q=\left\langle c, d \, \arrowvert \,
c^2=(cd)^2=d^2 \right\rangle$, of order 8, is a hyper-Hopfian group
(see \cite[Section~ 4]{D.HG}) and a coperfectly Hopfain group by Corollary
\ref{Nil}, which is not ultra-Hopfian (see \cite[Section~ 2]{VV1}). On the other hand, the solvable group of order
$p^4$ ($p$-prime), $$\left\langle x,y\, \arrowvert \,
x^{p^2}=y^{p^2}=1, \, y^{-1}xy=x^{1+p} \right\rangle$$ is not
hyper-Hopfian (see \cite[Section~ 4]{D.HG}), hence not ultra-Hopfian, but it is
coperfectly Hopfian by Theorem \ref{prop1}.

The group of rational numbers, $\mathbb{Q}$, is a coperfectly Hopfian group since it is Abelian and ultra-Hopfian (see \cite[Section~ 2]{VV1}).

The next lemma follows easily from  \cite[Theorem~ 2.10]{MKS}.

\begin{lemma} \label{New_1} Let $K_0$ be a free group on $k_0$ generators, and let $K_1$ and $K_2$ be nontrivial subgroups of $K_0$ such that $K_2 \unlhd K_1 \unlhd K_0$ and $K_2$ is finitely generated. Then, both $K_1$ and $K_2$ are free groups on $k_1$ and $k_2$ generators respectively, such that $[K_0 : K_2] < \infty$ and $k_0 \le  k_1 \le k_2 < \infty$, where $k_i = [K_{i-1} : K_i](k_{i-1}-1) + 1$ for $i = 1,  2$.
\end{lemma}

\begin{theorem} Every finitely generated free group is coperfectly Hopfian.
\end{theorem}

    \begin{proof} Let $F_n$ be a free group of $n$ generators, $n >1$ (note that $F_1=\mathbb{Z}$ is coperfectly Hopfian by Corollary \ref{Ab}). Let $f: F_n \to F_n$ be a homomorphism with $f(F_n) \unlhd F_n$ and $F_n/f(F_n)$ perfect. In this case $f(F_n) \ne 1$, since $F_n$ is not a perfect group. Then by Lemma \ref{New_1} it follows that $f(F_n)$ is a free group on $k$ generators, where $k=\left[ F_n: f(F_n) \right](n-1)+1 \ge n$ generators. This  can only occur when $\left[ F_n: f(F_n) \right]=1$, i.e., when $f$ is surjective. Since $F_n$ is Hopfian by \cite[Theorem~ 2.13]{MKS}, it follows that $f$ is an automorphism.
    \end{proof}

In  the next few results, we will be using some of the well-known properties of the fundamental group of a closed orientable surface $S$ of genus $g>1$, that we list here.  Recall that \[ \pi_1(S)=\left\langle a_1, b_1, \dots, a_g, b_g \left\arrowvert [a_1,b_1][a_2,b_2]\cdots [a_g,b_g]\right. \right\rangle,\]
     and $S$ has a cell structure with
    one 0-cell, $2g$ 1-cells, and one 2-cell. The 1-skeleton is a wedge
    sum of $2g$ circles and the 2-cell is attached along the loop given by the
    product of the commutators of these generators,
    $[a_1,b_1][a_2,b_2]...[a_g,b_g]$. By \cite[Proposition~ 2.45]{H} $\pi_1(S)$ is torsion-free, since $S$ is a 2-dimensional CW complex that is a $K(\pi_1(S),1)$ space by \cite[Example~ 1B.2]{H}. In addition, it is well known that $\pi_1(S)$ is not solvable (hence not Abelian). These groups are residually finite \cite{He} and finitely generated, hence Hopfian \cite{Mal}.

\begin{lemma} \label{New_2} Let $S$ be a closed orientable surface of genus $g_0 > 1$, and let $K_1$ and $K_2$ be nontrivial subgroups of $K_0 = \pi_1(S)$ such that $K_2 \unlhd K_1 \unlhd K_0$ and $K_2$ is finitely generated. Then, $[K_0 : K_2] < \infty$ and there exist $g_1,  g_2 \in \mathbb{Z}$ with $g_0 \le g_1 \le g_2$ such that $K_i\cong \pi_1(S_i)$ for $i = 1, 2$, where $S_i$ is a closed orientable surface of genus $g_i = [K_{i-1} : K_i](g_{i-1}-1)+1$.
\end{lemma}

\begin{proof} On the contrary, suppose that $[K_0 : K_1] = \infty$. By \cite[Corollary~
1]{HKS} and Lemma \ref{New_1}, both $K_1$ and $K_2$ are finitely generated free groups, which contradicts \cite[Theorem~ 6.1]{Gr}. Hence, $[K_0 : K_1] < \infty$, and the lemma now follows immediately from \cite[Corollary~ 3.1.9]{CGKZ} and
\cite[Theorem~ 6.1]{Gr}.
\end{proof}

\begin{theorem} \label{PH-S} Let $S$ be a closed orientable surface. Then $\pi_1(S)$ is coperfectly Hopfian.
\end{theorem}

    \begin{proof} Let $f: \pi_1(S) \to \pi_1(S)$ be a homomorphism with $f(\pi_1(S)) \unlhd \pi_1(S)$ and $\pi_1(S)/f(\pi_1(S))$ perfect. Note that $f(\pi_1(S)) \ne 1$ since $\pi_1(S)$ is not a perfect group. Then, by Lemma \ref{New_2} it follows that $\left[ \pi_1(S): f(\pi_1(S)) \right]< \infty$ and $f(\pi_1(S))$ is isomorphic to the fundamental group of a closed orientable surface of genus $g_1=\left[ \pi_1(S) : f(\pi_1(S)) \right](g-1)+1 \ge g$. This can only occur when $\left[ \pi_1(S) : f(\pi_1(S)) \right]=1$, i.e., when $f$ is surjective. Since $\pi_1(S)$ is Hopfian, it follows that $f$ is an automorphism.
    \end{proof}

Next we discuss another property among groups that we use later.

A group $G$ is {\it normally incommensurable with another group} $H$ if there is
no nontrivial homomorphism $f: G \to H$ such that $f(G) \trianglelefteq K \trianglelefteq H$ for some normal subgroup $K$ in $H$.

The proof of the next proposition follows easily from Lemma \ref{New_1}.

\begin{proposition} \label{fgfree} Let $F$ be a free group and $H$ be a finitely generated group with fewer generators than $F$. Then $H$ is normally incommensurable with $F$.
\end{proposition}

\begin{theorem} \label{fg} Let S be a closed orientable surface of genus $g > 1$ and $H$ be either a finitely generated virtually solvable group or a group on $h$ generators with $h < 2g$. Then, $H$ is normally incommensurable with
respect to $\pi_1(S)$.
\end{theorem}

\begin{proof} On the contrary, suppose that there exist a nontrivial homomorphism $f : H \to \pi_1(S)$ and a subgroup $K$ of $\pi_1(S)$ such that $f(H)\unlhd K \unlhd \pi_1(S)$. By Lemma \ref{New_2}, we may assume that $H$ is a finitely generated virtually solvable group. By the proof of \cite[Theorem 2.25]{M}, $f(H)$ has a normal solvable subgroup $L$ of finite index. Then Lemma \ref{New_2} and \cite[Corollary~ 3.1.9]{CGKZ}, imply that $L \cong \pi_1(S')$, where $S'$ is a closed orientable surface of genus $g' \ge g > 1$, a contradiction (since $\pi_1(S')$ cannot be solvable).
\end{proof}

\section{Products of Coperfectly Hopfian Groups} \label{PPHG}

  Next, we investigate when the property of being coperfectly Hopfian (discussed in Section \ref{PHG}) is preserved when taking free products and direct products of finitely generated  coperfectly Hopfian groups.

\begin{proposition} \label{free product} Let $G_1, G_2$ be nontrivial
finitely generated residually finite groups, $G_2 \ne
\mathbb{Z}_2$, and $G_1 \ast G_2$ not perfect. Then $G_1\ast G_2$ is a coperfectly Hopfian group.
\end{proposition}

       \begin{proof}
       Since $G_1\ast G_2$ is an ultra-Hopfian group by \cite[Theorem~ 2.2]{VV1}, it follows that
       $G_1\ast G_2$ is coperfectly Hopfian.
       \end{proof}

The following corollaries of Proposition \ref{free product} follow from
\cite[Corollaries~ 2.3, 2.4]{VV1} respectfully.

\begin{corollary}\label{free product 2} If $G_1$, $G_2$ are nontrivial
finitely generated groups such that $G_1$ is non-cyclic, and $G_1
\ast G_2$ is Hopfian and not perfect, then $G_1 \ast G_2$ is coperfectly Hopfian.
\end{corollary}

\begin{corollary}\label{free product 3} If $G_1$, $G_2$ are nontrivial
finitely generated, freely indecomposable coperfectly Hopfian groups, and $G_1$ is
non-cyclic, then $G_1 \ast G_2$ is coperfectly Hopfian.
\end{corollary}

Corollary \ref{free product 3} implies that under some particular conditions, the coperfectly Hopfian property is closed with
respect to free products.

Next we focus on direct products of coperfectly Hopfian groups. First, we need the following lemma.

\begin{lemma} \label{L1} Let $\phi :G_1 \times G_2 \to G_1 \times G_2$ be a
homomorphism. In addition,  let $i_{G_1}: G_1 \hookrightarrow G_1 \times G_2$, $i_{G_2}: G_2 \hookrightarrow G_1 \times G_2$ be the inclusions, and  $pr_{G_1}: G_1 \times G_2 \to G_1$, $pr_{G_2}: G_1 \times G_2 \to G_2$ be the projections onto the first and second factor respectively.

    \begin{enumerate}
    \item \label{L1.1} If $pr_{G_2} \circ \phi \circ i_{G_1}$ is trivial, then $\phi(G_1 \times 1) \subseteq G_1 \times 1$.
    \item \label{L1.2} If $pr_{G_2} \circ \phi \circ i_{G_1}$ is trivial and $pr_{G_2} \circ \phi \circ i_{G_2}$ is an isomorphism,  then $\phi(G_1 \times G_2) \cap (G_1\times 1) = \phi (G_1 \times 1) = \phi \circ i_{G_1} (G_1)$.
    \item If $\phi(G_1 \times G_2)$ is a normal subgroup of $G_1 \times G_2$, then $\phi(G_1 \times G_2) \cap (G_1\times 1)$ is  a normal subgroup of $G_1 \times 1$.
    \end{enumerate}
\end{lemma}

\begin{proof}
    \begin{enumerate}
    \item Since $pr_{G_2} \circ \phi \circ i_{G_1}$ is trivial, then $pr_{G_2} \circ \phi \circ i_{G_1}(G_1) = pr_{G_2} (\phi (G_1 \times 1))=1$, which implies that $\phi(G_1 \times 1) \subseteq G_1 \times 1$.
    \item By part \ref{L1.1}, it follows that $\phi(G_1 \times 1) \subseteq \phi(G_1 \times G_2) \cap (G_1 \times 1)$. We only need to prove that $\phi(G_1 \times G_2) \cap (G_1 \times 1) \subseteq \phi(G_1 \times 1)$.

        Let $x \in \phi(G_1 \times G_2) \cap (G_1 \times 1)$. Then $x=\phi(g_1,g_2) \in G_1 \times 1$ for some $(g_1,g_2) \in G_1 \times G_2$. Hence
        \[ \begin{array}{rl}
        x=&\phi(g_1,g_2)=\phi((g_1,e_{G_2})(e_{G_1},g_2))=\phi(g_1,e_{G_2})\phi(e_{G_1},g_2)\\
    =&(pr_{G_1}\circ \phi \circ i_{G_1}(g_1), e_{G_2})(pr_{G_1}\circ \phi \circ i_{G_2}(g_2),pr_{G_2}\circ \phi \circ i_{G_2}(g_2))\\
    =&(pr_{G_1}\circ \phi \circ i_{G_1}(g_1) pr_{G_1}\circ \phi \circ i_{G_2}(g_2), pr_{G_2}\circ \phi \circ i_{G_2}(g_2)) \in G_1 \times 1.
    \end{array}\]

    Hence $pr_{G_2}\circ \phi \circ i_{G_2}(g_2)=e_{G_2}$, and since $pr_{G_2} \circ \phi \circ i_{G_2}$ is an isomorphism, it follows that $g_2=e_{G_2}$, i.e., $x=\phi(g_1,e_{G_2}) \in \phi(G_1 \times 1)$.

    \item Since $\phi(G_1 \times G_2) \unlhd G_1 \times G_2$, it follows that \[ (g,e_{G_2}) \left(\phi( G_1 \times G_2) \cap (G_1 \times 1) \right)(g^{-1},e_{G_2}) \subseteq \phi(G_1 \times G_2)\] for all $g \in G_1$. Moreover,
        \[ (g,e_{G_2})(a,e_{G_2})(g^{-1},e_{G_2})=(gag^{-1},e_{G_2}) \in G_1 \times 1\] for all $a \in pr_{G_1}\left(\phi( G_1 \times G_2) \cap (G_1 \times 1) \right)$. Hence, \[ (g,e_{G_2}) \left(\phi( G_1 \times G_2) \cap (G_1 \times 1) \right)(g^{-1},e_{G_2}) \subseteq \phi(G_1 \times G_2) \cap (G_1 \times 1)\] for all $g \in G_1$.
        Therefore, $\phi(G_1 \times G_2) \cap (G_1\times 1) \unlhd G_1 \times 1$.
    \end{enumerate}\end{proof}

\begin{theorem} \label{PH} Let $G_1$, $G_2$ be coperfectly Hopfian groups such that
$G_1$ is normally incommensurable with $G_2$. Then $G_1 \times G_2$ is
coperfectly Hopfian.
\end{theorem}

    \begin{proof} Let $\phi :G_1 \times G_2 \to G_1 \times G_2$ be a
    homomorphism with $\phi (G_1 \times G_2) \trianglelefteq (G_1 \times G_2)$ and
    $\left( G_1 \times G_2 \right) / \phi(G_1 \times G_2)$ perfect. Since $G_1 \times G_2$ is not perfect, $\phi(G_1 \times G_2) \ne 1$. For $k=1, 2$, let $i_{G_k}: G_k \to G_1 \times G_2$ be the inclusion and $pr_{G_k}: G_1 \times G_2 \to G_k$ be the projection.

    First, we show that $pr_{G_2} \circ \phi \circ i_{G_2}:G_2 \to G_2$ is an isomorphism. Consider the map $pr_{G_2} \circ \phi \circ i_{G_1}: G_1 \to G_2$.  Using the fact that $\phi(G_1 \times G_2) \unlhd G_1 \times G_2$ and $pr_{G_2}$ is onto, it follows that $pr_{G_2} \circ \phi (G_1 \times G_2) \unlhd G_2$. Note that $\phi(G_1 \times 1) \unlhd \phi(G_1 \times G_2)$. Hence, $pr_{G_2}\circ \phi \circ i_{G_1}(G_1) \unlhd pr_{G_2} \circ \phi (G_1 \times G_2) \unlhd G_2$.
      Since $G_1$ is normally incommensurable with $G_2$, it follows that $pr_{G_2} \circ \phi \circ i_{G_1}$ is trivial. Then
     $pr_{G_2} \circ \phi (G_1 \times G_2) = pr_{G_2} \circ \phi \circ i_{G_2}(G_2)$. Hence, $pr_{G_2} \circ \phi \circ i_{G_2}(G_2)= pr_{G_2} \circ \phi (G_1 \times G_2) \unlhd G_2$. Since we have the epimorphism $$\widetilde{pr_{G_2}}: \left(G_1 \times G_2\right) /
    \phi(G_1 \times G_2) \to G_2/\left( pr_{G_2} \circ \phi(G_1 \times G_2)\right)$$ induced by $pr_{G_2}$, we see that $G_2/\left( pr_{G_2} \circ \phi(G_1 \times G_2)\right)$ is a perfect group. Now, the property of $G_2$ being coperfectly Hopfian, implies that
     $pr_{G_2} \circ \phi  \circ i_{G_2}: G_2 \to G_2$ is an isomorphism.

     Next we show that $pr_{G_1} \circ \phi \circ i_{G_1}: G_1 \to G_1$ is an isomorphism. Using the fact that $pr_{G_2} \circ \phi \circ i_{G_1}$ is trivial, $pr_{G_2} \circ \phi  \circ i_{G_2}$ is an isomorphism, and  $\phi(G_1 \times G_2)$ is a normal subgroup of $G_1 \times G_2$, by Lemma \ref{L1} it follows that  $\phi(G_1 \times G_2) \cap (G_1\times 1)=\phi (G_1 \times 1)=\phi \circ i_{G_1} (G_1)$ and $\phi(G_1 \times G_2) \cap (G_1\times 1)$ is  a normal subgroup of $G_1 \times 1$. Hence, $pr_{G_1} \circ \phi \circ i_{G_1}(G_1) \unlhd G_1$. In particular, $\left( G_1 \times 1 \right) / \left( \phi \circ i_{G_1}(G_1) \right) \cong G_1 / \left( pr_{G_1} \circ \phi \circ i_{G_1}(G_1) \right)$. Since we have the epimorphism $$\widetilde{pr_{G_1}}: (G_1 \times G_2) / \phi(G_1 \times G_2) \to G_1 / \left( pr_{G_1} \circ \phi(G_1 \times G_2) \right) \cong G_1 / \left( pr_{G_1} \circ \phi \circ i_{G_1} (G_1)\right)$$ induced by $pr_{G_1}$, we see that $ G_1 / \left( pr_{G_1} \circ \phi \circ i_{G_1} (G_1)\right)$ is a perfect group. Since $G_1$ is coperfectly Hopfian, it follows that $pr_{G_1} \circ \phi \circ i_{G_1}$ is an isomorphism.

    Since $pr_{G_k} \circ \phi \circ i_{G_k}$ is an isomorphism for $k=1, 2$ and $pr_{G_2} \circ \phi \circ i_{G_1}$is trivial, $\phi$ is an isomorphism.\end{proof}

\section{\bf Shape ${\rm \bf m_{simpl}o}$-Fibrators} \label{SF}

The following PL setting is used for the rest of this paper: let $N$ be a fixed closed PL $n$-manifold, $M$ a
(PL) $(n+k)$-manifold, $B$ a polyhedron, and $p:M \to B$ a
proper, surjective (PL) map. The map $p:M \to B$ is said to be an {\it $N$-shaped (PL) map} if each
fiber $p^{-1}(b)$, $b \in B$, has the homotopy type (or more generally the shape \cite{B, M-S}) of $N$.

The closed PL $n$-manifold $N$ is called a {\it codimension-$\it k$ shape ${\rm \it
m_{simpl}}$-fibrator} \cite{VV1} if for every closed PL $(n+k)$-manifold $M$
and $N$-shaped PL map $p:M \to B$, where $B$ is a simplicial
triangulated manifold, $p$ is an approximate fibration. Note that the abbreviation ${\rm m_{simpl}}$ points out that the target space is a
simplicial triangulated manifold. Similarly, the manifold  $N$ is a {\it codimension-$\it k$ shape orientable
${\rm \it m_{simpl}}$-fibrator} if for every closed
\underline{orientable} PL $(n+k)$-manifold $M$ and $N$-shaped PL map
$p:M \to B$, where $B$ is a simplicial triangulated manifold, $p$ is
an approximate fibration. We abbreviate this by writing that $N$ is
a codimension-$k$ shape ${\rm m_{simpl}o}$-fibrator. If $N$ is a codimension-$k$ shape ${\rm m_{simpl}}$-fibrator
(codimension-$k$ shape ${\rm m_{simpl} o}$ -fibrator) for all $k$,
then $N$ is called a {\it shape ${\rm \it m_{simpl}}$-fibrator}
({\it shape ${\rm \it m_{simpl}\it o}$-fibrator}).

Note that there cannot be much difference between
codimension-2 PL fibrators and codimension-2 PL shape ${\rm
m_{simpl}}$-fibrators, since the image spaces $B$ in codimension-2 are always manifolds by
\cite[Theorem~3.6]{D.W.}. The two classes are precisely the same among
Hopfian manifolds with Hopfian fundamental groups.

Let $p:M \to B$ be an $N$-shaped PL map. The {\it continuity set} of
$p$, $C$, consists of all points $b \in B$, such
that under any retraction $R:p^{-1}U \to p^{-1}b$ defined over a
neighborhood $U\subset B$ of $b$, $b$ has another neighborhood $V_b
\subset U$, such that for all $x \in V_b$, $R\arrowvert :p^{-1}x \to
p^{-1}b$ is a degree one map. Establishing that $p$ is an approximate fibration, usually requires one to prove that the target space $B$ equals the continuity set of $p$, as the next lemma shows. Note that this lemma follows immediately from the definitions and Coram and Duvall's characterization of approximate fibrations in terms of movability properties \cite[Proposition~3.6]{Coram 2}.

\begin{lemma} \label{AF1} Let $N$ be a Hopfian $n$-manifold with a Hopfian fundamental group and $p:M \to B$ be an $N$-shaped PL map, where $M$ is a closed orientable PL $(n+k)$-manifold, and $B$ is a
triangulated manifold. Then the continuity set of $p$, $C$, is equal
to $B$ if and only if $p$ is an approximate fibration over $B$.
\end{lemma}

The next few results listed below are needed for the proof of the main theorem.

\begin{lemma} \cite[Lemma~4.1]{VV2} \label{surjection} Let $N$ be a Hopfian $n$-manifold and $p: M \to \mathbb{R}^k$, $k \ge 2$, be an $N$-shaped PL map from an open orientable PL $(n+k)$-manifold. Suppose $T \subset \mathbb{R}^k$ is a closed set with
$\textrm{dim } T \le k-2$. Then $j_{\sharp}: \pi_1 \big( p^{-1}(\mathbb{R}^k \backslash T ) \big) \to \pi_1 \big(
p^{-1}(\mathbb{R}^k) \big)$ is surjective, where $j:p^{-1}(\mathbb{R}^k \backslash T ) \to p^{-1}(\mathbb{R}^k)$ is
the  inclusion map.
\end{lemma}

The next result (that we use later) and its proof is the analog to the Fundamental Theorem \cite[Theorem~ 5.5]{VV1} and its proof.

\begin{theorem} \label{big theorem}Let $N$ be a closed orientable PL $n$-manifold, homotopically determined by $\pi_1$ with a coperfectly Hopfian fundamental group and $p: M \to \mathbb{R}^k$, $k>2$,
be an $N$-shaped PL map from an  open orientable PL
$(n+k)$-manifold. Suppose $T \subset \mathbb{R}^k$ is closed with
$\textrm{dim } T < k-2$, and such that $p \arrowvert _{p^{-1}\big(
\mathbb{R}^k \backslash T \big)}$ is an approximate fibration. Then
$p$ is an approximate fibration.
\end{theorem}

\begin{proof} Let $T \subset \mathbb{R}^k$ be closed with $\textrm{dim } T < k-2$.
Without loss of generality we can assume that $T$ is a minimal closed set
such that $p \arrowvert _{p^{-1}\big( \mathbb{R}^k \backslash T
\big)}$ is an approximate fibration.

On the contrary, suppose that $T \ne \emptyset$. Since $T$ is a closed subset of $\mathbb{R}^k$ and $p$ is an $N$-shaped map, by Daverman and Husch's work on decompositions and approximate
fibrations \cite{Daverman-Hush}, there exist $W \subset \mathbb{R}^k$, $t \in W \cap T$ and a retraction $R : p^{-1}(W) \to p^{-1}(t)$
such that $W \approx \mathbb{R}^k$, and $R\arrowvert : p^{-1}(s) \to p^{-1}(t)$ is a homotopy equivalence
for any $s \in W \cap T$. Fix $x \in  W \backslash T$. By assumption and the minimality of $T$,
it suffices to show that $R\arrowvert_{\sharp}: \pi_1(p^{-1}(x)) \to  \pi_1(p^{-1}(t))$ is an isomorphism.

Using the fact that $p$ is an approximate fibration over $W\backslash T$, the
homotopy exact sequence
\[ \pi_1 \big( p^{-1}(x) \big) \cong \pi_1 (N) \stackrel{i_{\sharp}}
{\longrightarrow} \pi_1 \big( p^{-1}(W \backslash T) \big)
\stackrel{{p \arrowvert}_{\sharp}} {\longrightarrow} \pi_1 \big( W
\backslash T \big) \longrightarrow 1 \cong \pi_0(N)\]  gives
\begin{equation}\label{1a}
i_{\sharp} \big( \pi_1(N) \big)= {\rm ker} \, p_{\sharp} \unlhd \pi_1
\big( p^{-1}(W \backslash T) \big).
\end{equation}
Hence,
 \begin{equation} \label{1}
 \pi_1\left( p^{-1}(W\backslash T)\right) / i_{\sharp}(\pi_1(N)) \cong p_{\sharp}\left( p^{-1}(W\backslash T)\right) = \pi_1(W \backslash T).
 \end{equation}
In addition, by Lemma \ref{surjection} it follows that the map $j_{\sharp}: \pi_1 \big( p^{-1}(W \backslash T) \big) \rightarrow \pi_1 (p^{-1}W)$ is onto.

Next, look at the long exact homology sequence of the pair $\big( W, W \backslash
T \big)$:
\[
\begin{array}{ccccccccc}
\cdots \to &H_2 (W) &\longrightarrow &H_2\big( W,W \backslash T
\big)&
\longrightarrow& H_1 \big( W \backslash T \big) &\longrightarrow &H_1 ( W )&\to \cdots\\
&\Big\downarrow \vcenter {%
\rlap{$\scriptstyle{\cong}$}}&&&&&&\Big\downarrow \vcenter {%
\rlap{$\scriptstyle{\cong}$}}&\\
& H_2(\mathbb{R}^k)&&&&&& H_1(\mathbb{R}^k)
\end{array}
\]
Hence, $H_1(W \backslash T) \cong H_c^{k-2}(T\cap W)$ by Alexander duality \cite[p.~342]{Spanier}. Since $ H_c^{k-2}(T\cap W) \cong 0$ (${\rm
dim}(T\cap W)\le \textrm{dim } T <k-2$ and $k>2$), it follows that $H_1 (
W \backslash T) \cong 0$, hence $\pi_1( W \backslash T)$ is perfect.

Consider the following diagram:
\[
\xymatrix{
 \pi_1 (N) \cong \pi_1(p^{-1}(x)) \ar[r]^{i_{\sharp}} \ar[ddr]_{R \arrowvert_{\sharp}}& \pi_1
\big( p^{-1}(W \backslash T) \big)
\ar[r]^{p\arrowvert_{\sharp}} \ar[d]^{j_{\sharp}}& \pi_1 (W \backslash T) \ar[r]  &1\\
& \pi_1 (p^{-1}W)\ar[d]^{R_{\sharp}} &&\\
&\pi_1 (N)\cong \pi_1(p^{-1}(t))&&}\]
Since $j_{\sharp}$ and $R_{\sharp}$ are surjective, $R\arrowvert_{\sharp}(\pi_1(N)) = R_{\sharp} \circ  j_{\sharp}(i_{\sharp}(\pi_1(N))) \trianglelefteq \pi_1(N)$.
Note that the map
\[ \widetilde{R_{\sharp} \circ j_{\sharp}} : \pi_1(p^{-1}(W \backslash T)) / i_{\sharp}(\pi_1(N)) \to  \pi_1(N)/R \arrowvert_{\sharp}(\pi_1(N))\]
induced by $R_{\sharp} \circ j_{\sharp}$ is an epimorphism. Since $\pi_1(W \backslash T)$ is perfect, by (\ref{1}),
so is $\pi_1(N)/R \arrowvert_{\sharp}(\pi_1(N))$. Hence $\pi_1(N)$ being coperfectly Hopfian implies
that $R\arrowvert_{\sharp}$ is an isomorphism, which proves the theorem.
\end{proof}

Next we prove the main theorem in this section.

\begin{theorem} \label{Fib} Let $N$ be a closed orientable manifold homotopically determined by
$\pi_1$ with a coperfectly Hopfian fundamental group. If $N$ is a
codimension-2 shape m$_{simpl}$o-fibrator, then $N$ is a shape
m$_{simpl}$o-fibrator.
\end{theorem}

\begin{proof} Suppose $N$ is a codimension-$(k-1)$ shape
${\rm m_{simpl}o}$-fibrator. Assume $p: M^{n+k} \to B$ is an $N$-shaped PL map, where $M$ is a
closed orientable PL $(n+k)$-manifold and $B$ is a triangulated
manifold. Then by \cite[Lemma~8.1]{VV1} $p$ is an approximate
fibration over $B\backslash B^{(k-3)}$, and Theorem \ref{big
theorem} implies that $p$ is an approximate fibration.

By induction on $k$, $N$ is a codimension-$k$ shape ${\rm
m_{simpl}o}$-fibrator.
\end{proof}

\noindent {\bf Remark:}  The condition of the manifold $N$ being a codimension-2 fibrator cannot be omitted. Namely, take an $n$-dimensional torus, $T$. $T$ is aspherical, so is homotopically determined by $\pi_1$. Furthermore, $\pi_1(T)$ is a finitely generated Abelian  group (so coperfectly Hopfian by Corollary \ref{Ab}). But $T$ is not a codimension-2 fibrator.

\begin{corollary} All closed orientable  surfaces $S$ with genus $g >1$ are shape m$_{simpl}$o-fibrators.
\end{corollary}

\begin{proof} $S$ is homotopically determined by $\pi_1$, has coperfectly Hopfian fundamental group  by Theorem \ref{PH-S}, and is a codimension-2 shape ${\rm m_{simpl}o}$-fibrator by \cite[Corollary~ 6.3]{VV1}. So by Theorem \ref{Fib} it follows that $S$ is a shape
m$_{\rm simpl}$o-fibrator.
\end{proof}

\section{\bf Shape Fibrator's Properties of Direct Products of Hopfian Manifolds} \label{PSF}

In this section we discuss the shape m$_{\rm simpl}$o-fibrator's properties of direct product of Hopfian manifolds. First we list a lemma that follows directly from the proof of \cite[Theorem~4.1]{Daverman-Kim}. For completeness we include the proof here.

\begin{lemma} \label{D-K} (\cite[Theorem~4.1]{Daverman-Kim}) Suppose $N_1$, $N_2$ are closed orientable manifolds of dimensions $m$ and $n$ respectively, $m \ne n$, homotopicaly determined by $\pi_1$. Assume also that $N_1 \times N_2$ is a Hopfian manifold, $\pi_1(N_1)$ is normally incommensurable with $\pi_1(N_2)$, and both $\pi_1(N_1)$ and $\pi_1(N_2)$ are coperfectly Hopfain. Then $N_1 \times N_2$ is homotopically determined by $\pi_1$.
\end{lemma}

\begin{proof}
Let $\phi: N_1 \times N_2 \to N_1 \times N_2$ be a map that induces a $\pi_1$-isomorphism. For $e=1, 2$, let $i_e: N_e \to N_1 \times N_2$ be the inclusion and $pr_e:N_1 \times N_2 \to N_e$ be the projection.

First we show that $(pr_2 \circ \phi \circ i_2)_{\sharp}$ and $(pr_1 \circ \phi \circ i_1)_{\sharp}$ are isomorphisms. Using the fact that $\phi_{\sharp}$ is onto, it follows that  $(pr_2 \circ \phi \circ i_1)_{\sharp}(\pi_1(N_1)) \trianglelefteq \pi_2(N_2)$ which together with the normal incommensurability of $\pi_1(N_1)$ with $\pi_1(N_2)$, implies that $(pr_2 \circ \phi \circ i_1)_{\sharp}$ is trivial. This implies that $(pr_2 \circ \phi \circ i_2)_{\sharp}(\pi_1(N_2))=\pi_1(N_2)$ since $\phi_{\sharp}$ is onto. Now, using the hypothesis that $\pi_1(N_2)$ is Hopfian, $(pr_2 \circ \phi \circ i_2)_{\sharp}$ is an isomorphism. Since $\phi_{\sharp}$ is onto, by Lemma \ref{L1} (\ref{L1.2}) it follows that $(pr_1 \circ \phi \circ i_1)_{\sharp}(\pi_1(N_1))=\pi_1(N_1)$. Then the Hopfian property of $\pi_1(N_1)$ implies that $(pr_1 \circ \phi \circ i_1)_{\sharp}$ is an isomorphism.

Hence, $pr_e \circ \phi \circ i_e: N_e \to N_e$, $e=1, 2$, induces a $\pi_1$-isomorphism, and therefore by hypothesis is a homotopy equivalence. Choose generators, $\eta$, $\eta'$ of $H_m(N_1)$, $H_n(N_2)$, respectively. Since $pr_e \circ \phi \circ i_e$, $e=1, 2$, induces a homology isomorphism, $(pr_1 \circ \phi \circ i_1)_*(\eta)$ is a generator of $H_m(N_1)$ and $(pr_2 \circ \phi \circ i_2)_*(\eta')$ is a generator of $H_n(N_2)$. Without loss of generality we can assume that $m>n$. Then $(pr_2 \circ \phi \circ i_1)_*(\eta)=0$. An application of the K\"{u}nneth Theorem gives
\[\begin{array}{lcl}
\phi_*(\eta \otimes \eta') &=& (pr_1 \circ \phi \circ i_1)_*(\eta) \otimes (pr_2 \circ \phi \circ i_2)_*(\eta')\\
&&+ (pr_1 \circ \phi \circ i_2)_*(\eta') \otimes (pr_2 \circ \phi \circ i_1)_*(\eta)\\
&=& (pr_1 \circ \phi \circ i_1)_*(\eta) \otimes (pr_2 \circ \phi \circ i_2)_*(\eta').
\end{array}\]
Hence, $\phi_*(\eta \otimes \eta')= (pr_1 \circ \phi \circ i_1)_*(\eta) \otimes (pr_2 \circ \phi \circ i_2)_*(\eta')$ is a generator of $H_m(N_1) \otimes H_n(N_2) \cong H_{m+n}(N_1 \otimes N_2)$. Therefore, $\phi$ is a degree one map, and $N_1 \times N_2$ being Hopfian implies that $f$ is a homotopy equivalence.
\end{proof}

The following theorems are immediately seen from Theorems \ref{PH} and \ref{Fib}, and Lemma \ref{D-K}.

\begin{theorem} \label{Dir} Suppose $N_1$, $N_2$ are closed orientable manifolds of dimension $m$ and $n$ respectively, $m \ne n$,
homotopically determined by $\pi_1$. Assume also that $N_1 \times N_2$
is a Hopfian manifold. In addition, $\pi_1(N_1)$ is normally incommensurable
with $\pi_1(N_2)$ and $\pi_1(N_1)$, $\pi_1(N_2)$ are
coperfectly Hopfian.

If $N_1 \times N_2$ is a codimension-2 shape m$_{simpl}$o-fibrator, then $N_1 \times N_2$ is a shape m$_{simpl}$o-fibrator.
\end{theorem}

\begin{theorem} \label{C-Dir} Suppose $N_1$, $N_2$ are closed orientable aspherical manifolds. In addition, assume that $\pi_1(N_1)$ is normally incommensurable
with $\pi_1(N_2)$ and $\pi_1(N_1)$, $\pi_1(N_2)$ are
coperfectly Hopfian.

If $N_1 \times N_2$ is a codimension-2 shape m$_{simpl}$o-fibrator, then $N_1 \times N_2$ is a shape m$_{simpl}$o-fibrator.
\end{theorem}

 \noindent{\bf Remark:} Note again the necessity of the requirement for the manifold $N_1 \times N_2$ to be a codimension-2 fibrator. Namely, take an $n$-dimensional torus $T$ and a closed orientable surface $S$ with genus $g>1$ such that $n < 2g$. They are both closed aspherical manifolds with coperfectly Hopfian fundamental groups by Corollary \ref{Ab} and Theorem \ref{PH-S}. By Theorem \ref{fg}, $\pi_1(T)$ is normally incommensurable with $\pi_1(S)$. But the manifold $T \times S$ is not a codimension-2 fibrator.

 \begin{example} Let $S_1$ and $S_2$ be two closed orientable surfaces with genuses $g_1$ and $g_2$ respectively with $g_2 > g_1>1$. Then $S_1$, $S_2$ are aspherical with coperfectly Hopfian fundamental groups  by Theorem \ref{PH-S}, and $\pi_1(S_1)$ is normally incommensurable with $\pi_1(S_2)$  by Theorem \ref{fg}. Since $S_1 \times S_2$ is a codimension-2 orientable fibrator by \cite[Main Theorem p.~9]{Yo}, $S_1 \times S_2$ is a shape m$_{\rm simpl}$o-fibrator by Theorem \ref{C-Dir}.

 \end{example}

\begin{example} Let $M^3$ be a closed orientable 3-manifold with Sol geometry that fibers over $S^1$ by \cite[Theorem~5.3]{Scott}. It is known that $M^3$ is aspherical, so homotopically determined by $\pi_1$.
Take $S$ to be a closed orientable surface with genus $g >1$. Then $M^3 \times S$ is aspherical as a product of aspherical manifolds, hence homotopically determined by $\pi_1$ and Hopfian.

 It is known that $\pi_1(M^3)$ is a finitely generated Hopfian solvable group, and is a hyper-Hopfian group by \cite[Theorem~7.2]{D.3-man}, hence coperfectly Hopfian by Theorem \ref{prop1}.   Since $\pi_1(M^3)$ is normally incommensurable with $\pi_1(S)$ by Theorem \ref{fg}, the proof of \cite[Lemma~5.1]{Daverman-Im-Kim} shows that $\pi_1(M^3\times S)$ is also hyper-Hopfian. Hence, $M^3 \times S$ is a codimension-2 fibrator by \cite[Theorem~5.4]{D.HG}.

$\pi_1(S)$ is coperfectly Hopfian by Theorem \ref{PH-S},
 hence, Theorem \ref{C-Dir} implies that $M^3 \times S$ is a shape m$_{\rm simpl}$o-fibrator.

\end{example}

\section{\bf Acknowledgment}

 The author cannot express enough thanks to Robert J. Daverman for his continued support, encouragement, and guidance during this project. The author would also like to thank Zoran $\check{\textrm{S}}$uni$\acute{\textrm{c}}$ for his valuable suggestions and Susan Hermiller for the useful discussions about the group properties. In addition, the author thanks the anonymous reviewer for the helpful suggestions and comments that benefited the paper tremendously.

\end{document}